\documentclass[runningheads,envcountsame]{llncs}

\usepackage[T1]{fontenc}
\usepackage{amsmath}
\usepackage{amssymb}
\usepackage{mathtools}
\usepackage{graphicx}
\usepackage{bbm}
\usepackage{xparse}
\usepackage{xcolor}
\usepackage{subcaption}
\usepackage{enumitem}
\usepackage{thm-restate}
\usepackage{todonotes}
\usepackage{tikz}
\usepackage{pgfplots}
\usepackage{tikz-3dplot}
\usetikzlibrary{shapes.geometric}
\pgfplotsset{compat=1.18}

\usepackage[colorlinks=true,citecolor=blue,urlcolor=blue]{hyperref}
\usepackage[capitalise]{cleveref}

\makeatletter
\renewcommand{\@Opargbegintheorem}[4]{%
  #4\trivlist\item[\hskip\labelsep{#3#2\@thmcounterend}]}
\makeatother

\DeclareDocumentCommand\R{}{\mathbb{R}}
\DeclareDocumentCommand\Rnonneg{}{\R_{\geq 0}}
\DeclareDocumentCommand\Z{}{\mathbb{Z}}
\DeclareDocumentCommand\N{}{\mathbb{N}}

\DeclareDocumentCommand\zerovec{o}{\IfNoValueTF{#1}{\mathbb{O}}{\mathbb{O}_{#1}}}
\DeclareDocumentCommand\onevec{o}{\IfNoValueTF{#1}{\mathbbm{1}}{\mathbbm{1}_{#1}}}
\DeclareDocumentCommand\conv{o}{\operatorname{conv}\IfValueTF{#1}{\left(#1\right)}{}}
\DeclareDocumentCommand\cl{o}{\operatorname{cl}\IfValueTF{#1}{\left(#1\right)}{}}
\DeclareDocumentCommand\aff{o}{\operatorname{aff}\IfValueTF{#1}{\left(#1\right)}{}}
\DeclareDocumentCommand\relint{o}{\operatorname{relint}\IfValueTF{#1}{\left(#1\right)}{}}
\DeclareDocumentCommand\spn{o}{\operatorname{span}\IfValueTF{#1}{\left(#1\right)}{}}
\DeclareDocumentCommand\proj{oo}{\IfValueTF{#1}{\operatorname{proj}{}_{#1}}{%
  \operatorname{proj}{}}\IfValueTF{#2}{\left(#2\right)}{}}
\DeclareDocumentCommand\transpose{m}{#1^{\intercal}}

\DeclareDocumentCommand\RLTstrong{}{\ensuremath{\widehat{\mathcal{R}}}}
\DeclareDocumentCommand\RLTweak{}{\ensuremath{\mathcal{R}}}
\DeclareDocumentCommand\RLTstrongEF{}{\ensuremath{\widehat{\mathcal{R}}
^{\textup{ext}}}}
\DeclareDocumentCommand\RLTweakEF{}{\ensuremath{\mathcal{R}
^{\textup{ext}}}}
\DeclareDocumentCommand\LandP{}{\ensuremath{\mathcal{L}}}

\title{Geometry of the Reformulation-Linearization-Technique: Domination of Disjunctions}
\titlerunning{Geometry of RLT: Domination of Disjunctions}
\author{Hugo A.~Hof\inst{1}\orcidID{0009-0009-7953-1763} \and
Matthias Walter\inst{1}\orcidID{0000-0002-6615-5983}}
\authorrunning{H.A.~Hof and M.~Walter}
\institute{%
  University of Twente, Enschede, The Netherlands\\
  \email{\{h.a.hof,m.walter\}@utwente.nl}
}

\begin{document}

\maketitle

\begin{abstract}
  The reformulation-linearization-technique (RLT) is a well-known strengthening technique for binary mixed-integer optimization.
  It is well known to dominate lift-and-project strengthening, which is based on disjunctive programming (DP) for single-variable disjunctions.
  In contrast to the latter, the geometry of RLT is not understood completely.
  We provide some insights by characterizing the points in the corresponding \emph{RLT closure} geometrically.
  We exploit this insight to show that RLT even dominates DP approaches based on cardinality equations with right-hand side $1$.
  This is in contrast to cardinality inequalities with right-hand side $1$, whose DPs are not dominated.
  Our results have applications in the strength comparison for the quadratic assignment problem.

  \keywords{reformulation-linearization-technique \and disjunctive programming \and lift-and-project \and mixed-integer programming \and quadratic assignment problem}

\end{abstract}

\section{Introduction}
\label{sec_intro}

Since its invention by Adams and Sherali~\cite{AdamsS86,AdamsS90}, the \emph{reformulation-linearization-technique (RLT)} is among the state-of-the-art (re)formulation techniques in mixed-integer (nonlinear) optimization.
It turned out to be important in particular for quadratic versions of combinatorial optimization problems such as the \emph{quadratic assignment problem}~\cite{FriezeY83,AdamsJ94} and the \emph{quadratic knapsack problem}~\cite{ForresterAH10}, as well as in polynomial optimization in general~\cite{QiuY24,SheraliT92}.
In this paper we consider binary mixed-integer programs, which are given by a polyhedron $P \subseteq [0,1]^N$ (for some index set $N$) that is defined by inequalities $Ax \leq b$, and a subset $B \subseteq N$ of indices whose variables are restricted to be binary.
In this setting, RLT entails the multiplication of the inequalities with each variable $x_j$ and its complement $1-x_j$, followed by a substitution of products $x_i \cdot x_j$ by auxiliary variables $y_{i,j}$, and finally the identification of $y_{j,j}$ with $x_j$ for all binary variables $x_j$ due to $x_j^2 = x_j$.
By multiplying with all combinations of $k \in \N$ such terms, this approach was generalized to higher \emph{levels} of the \emph{RLT hierarchy}~\cite{SheraliA99}.
Recently, a partial application of this procedure has been successfully implemented in general-purpose mixed-integer programming solvers~\cite{BestuzhevaGA23} in order to improve performance in presence of bilinear product relations.

RLT is closely related to other strengthening techniques such as \emph{disjunctive programming}.
Here, one considers polyhedra $P_1,P_2,\dotsc,P_k \subseteq P$ that cover all feasible solutions of the optimization problem.
A celebrated result by Balas~\cite{Balas74,Balas79} establishes a description of the convex hull $\conv(P_1 \cup P_2 \cup \dotsb \cup P_k)$ of their union by means of an \emph{extended formulation}, that is, using additional variables.
As a special case, the strengthening of $P$ by considering so-called \emph{lift-and-project} inequalities, which are inequalities valid for $\conv \big(P^{(j)}_0 \cup P^{(j)}_1 \big)$ for $P^{(j)}_\beta \coloneqq \{ x \in P \mid x_j = \beta \}$ was proposed by Balas, Ceria and Cornu{\'e}jols~\cite{BalasCC93,BalasCC96}.
It is known that all such inequalities are implied by RLT~\cite{AdamsS15}.
The strengthening obtained by lift-and-project is geometrically well understood.
However, this is not the case for RLT, and our paper contributes to this understanding.

\paragraph{Contribution and outline.}
We start by formally introducing RLT in two variants (with and without multiplication of continuous variables), disjunctive programming and lift-and-project as well as their extended formulations in \cref{sec_strengthen}.

In \cref{sec_geometry} we recap some well-known geometric properties of RLT, from which we derive a characterization that gives new geometric insights.
The characterization describes points in the feasible region as convex combinations of points on the faces of the unit cube, together with an additional constraint.

Then, in \cref{sec_qap} we consider the \emph{quadratic assignment problem} (QAP) first studied in~\cite{KoopmansB57}. 
We prove the dominance of the well-known Adams-Johnson linearization~\cite{AdamsJ94}, which can be derived via RLT, over a disjunctive programming approach that is more powerful than lift-and-project.
As a corollary, we obtain dominance over several well-known linearizations, which has, to the best of our knowledge, only been settled~\cite{HuberR17} for the Xia-Yuan linearization~\cite{XiaY06}.

Our QAP-specific observations are generalized in \cref{sec_dominance}, where we characterize which disjunctive programming approaches are generally dominated by RLT.
This characterization describes for all disjunctions based on enumeration of binary solutions of a subset of variables, whether RLT dominates the disjunctive programming approach for such a disjunction.

We conclude our paper with several open problems, in particular to complete the geometric understanding of RLT beyond our contribution.

\section{Strengthening techniques}
\label{sec_strengthen}

Consider a polyhedron $P \subseteq [0,1]^N$ defined by inequalities $Ax \leq b$ and an index set $B \subseteq N$ of binary variables.
The \emph{reformulation linearization technique (RLT)} consists in multiplying constraints with variables or their complements, replacing products $x_i \cdot x_j$ (for $i,j \in N$) by variables $y_{i,j}$, and substituting $y_{j,j} = x_j$ for $j \in B$.
Traditionally, one multiplies all constraints with all binary variables $x_j$ and with their complements $1-x_j$ for all $j \in B$, only.
We define the corresponding extended formulation $\RLTweakEF_B(P)$ as the set of vectors $(x,y) \in \R^N \times \R^{N \times B}$ that satisfy 
\begin{subequations}
  \label{eq_rlt_weak}
  \begin{alignat}{7}
    A y_{\star,j} &\leq b x_j  &\quad& \forall j \in B, \label{eq_rlt_weak_first} \\
    A (x - y_{\star,j}) &\leq b (1 - x_j)  &\quad& \forall j \in B, \label{eq_rlt_weak_second} \\
    y_{i,j} &= y_{j,i} &\quad& \forall i,j \in B, \label{eq_rlt_weak_symmetry} \\ 
    y_{j,j} &= x_j &\quad& \forall j \in B. \label{eq_rlt_weak_diag}
  \end{alignat}
\end{subequations}
Note that by $y_{\star,j}$ we denote the vector with entries $y_{i,j}$ for all $i$.
If we additionally multiply with $x_j$ and $1-x_j$ for $j \in N \setminus B$, we obtain the set $\RLTstrongEF_B(P)$ of vectors $(x,y) \in \R^N \times \R^{N \times N}$ that satisfy
\begin{subequations}
  \label{eq_rlt_strong}
  \begin{alignat}{7}
    A y_{\star,j} &\leq b x_j  &\quad& \forall j \in N, \label{eq_rlt_first} \\
    A (x - y_{\star,j}) &\leq b (1 - x_j)  &\quad& \forall j \in N, \label{eq_rlt_second} \\
    y_{i,j} &= y_{j,i} &\quad& \forall i,j \in N, \label{eq_rlt_symmetry} \\ 
    y_{j,j} &= x_j &\quad& \forall j \in B. \label{eq_rlt_diag}
  \end{alignat}
\end{subequations}
The corresponding \emph{RLT closures} (with respect to $B$), $\RLTstrong_B(P)$ and $\RLTweak_B(P)$ are defined as the projections of $\RLTstrongEF_B(P)$ and $\RLTweakEF_B(P)$, respectively.
Our notation suggests that these polyhedra are independent of the inequality description of $P$, which we confirm later in \cref{thm_rlt_elementary}.
Note that in some papers, additional \emph{McCormick inequalities}~\cite{Fortet60a,Fortet60b,GloverW73,GloverW74,McCormick76} are part of the description.
In the same proposition we will show that these are implied due to our assumption that $P$ lies in $[0,1]^N$.

\paragraph{Disjunctive programming.}
In line with~\cite{Balas74,Balas79} we define a \emph{disjunctive constraint} as a logical OR
\begin{equation}
  \bigvee\limits_{h \in H} (C^{(h)} x \leq d^{(h)}) \label{eq_disjunctive_cons}
\end{equation}
of constraint sets $C^{(h)} x \leq d^{(h)}$ for some finite index set $H$.
The corresponding \emph{disjunctive hull for a polyhedron $P$} is
\begin{equation}
  \cl \conv \Big( \bigcup_{h \in H} \big\{ x \in P \mid C^{(h)}x \leq d^{(h)} \big\} \Big), \label{eq_disjunctive_hull}
\end{equation}
where $\cl(\cdot)$ denotes the \emph{topological closure}.
By slightly abusing notation we also allow equality constraints as in the \emph{(binary) variable disjunction} $(x_j = 0) \lor (x_j = 1)$.
Note that all our considered polyhedra will be polytopes, which allows us to omit the topological closure in~\eqref{eq_disjunctive_hull}~\cite{Balas74}.
A well-known result due to Balas establishes the \emph{corresponding extended formulation} for~\eqref{eq_disjunctive_hull} with additional variables $x^{(h)} \in \R^N$ and $\lambda^{(h)} \in \R$ for all $h \in H$ and constraints
\begin{subequations}
  \label{model_disjunctive_ef}
  \begin{alignat}{7}
    A x^{(h)} & \leq b \lambda^{(h)} &\quad& \forall h \in H, \label{eq_dp_ext_first}\\
    C^{h} x^{(h)} &\leq d^{(h)} \lambda^{(h)} &\quad& \forall h \in H, \label{eq_dp_ext_second}\\
    \sum_{h \in H} \lambda^{(h)} & = 1, \label{eq_dp_ext_third}\\
    \sum_{h \in H} x^{(h)} &= x. \label{eq_dp_ext_fourth}
  \end{alignat}  
\end{subequations}

\begin{proposition}[Balas~\cite{Balas74}]
  The disjunctive hull~\eqref{eq_disjunctive_hull} for a disjunctive constraint~\eqref{eq_disjunctive_cons} is equal to the projection of the polyhedron defined by~\eqref{model_disjunctive_ef} onto the $x$-variables.
\end{proposition}

\paragraph{Lift-and-project.}
The \emph{lift-and-project closure of $P$} with respect to binary variables $B \subseteq N$ is defined as the intersection of the disjunctive hulls for variable disjunctions for all $j \in B$, that is,
\begin{equation*}
  \LandP_B(P) \coloneqq \bigcap_{j \in B} \conv \Big( \big\{ x \in P \mid x_j = 0 \big\} \cup \big\{ x \in P \mid x_j = 1 \big\} \Big).
\end{equation*}

 We conclude this section with a well-known result for the relationship between $\RLTweak_B(P)$ and $\LandP_B(P)$ that, for example, is shown in~\cite{AdamsS15}, and whose proof we provide for completeness.

 \begin{proposition}
  \label{thm_rlt_lap_dominance}
  For polyhedra $P\subseteq[0,1]^N$ and $B \subseteq N$, $\RLTweak_B(P) \subseteq \LandP_B(P)$ holds.
\end{proposition}

\begin{proof}
  Assume $P = \{ x \in \R^N \mid Ax \leq b \}$ for some matrix $A \in \R^{M \times N}$ and $b \in \R^M$.

  Consider a point $\hat{x} \in \RLTweak_B(P)$ and some $\hat{y} \in \R^{N \times B}$ such that~\eqref{eq_rlt_weak} holds. 
  Take some $j \in B$ and consider the disjunction $x_j = 0 \lor x_j = 1$, that is, with constraints $x_j = h$ for both $h \in H \coloneqq \{0,1\}$.
      
  Let $x^{(1)} = \hat{y}_{\star,j}$, $x^{(0)} = \hat{x} -\hat{y}_{\star,j}$, $\lambda^{(1)} = \hat{x}_j$ and $\lambda^{(0)} = 1-\hat{x}_j$. 
  Then from~\eqref{eq_rlt_weak_first} and~\eqref{eq_rlt_weak_second} we have $Ax^{(1)} \leq b\hat{x}_j = b \lambda^{(1)}$ and $Ax^{(0)} \leq b (1 - \hat{x}_j ) = b \lambda^{(0)}$ which means that~\eqref{eq_dp_ext_first} holds.
  From~\eqref{eq_rlt_weak_diag} we get $x^{(1)}_j = \hat{x}_j = 1 \cdot \lambda^{(j)}$ and $x^{(0)}_j = \hat{x}_j - \hat{x}_j = 0 \cdot \lambda^{(0)}$, so~\eqref{eq_dp_ext_second} holds.
  Lastly, \eqref{eq_dp_ext_third} and~\eqref{eq_dp_ext_fourth} follow by our definition of $x^{(h)}$ and $\lambda^{(h)}$. 
  Therefore, $\hat{x} \in \conv ( \{ x \in P \mid x_j = 0 \} \cup \{ x \in P \mid x_j = 1 \} )$.
  As this holds for an arbitrary $j$ it also holds for the intersection of all $j \in B$.
  \qed
\end{proof}
      
\section{Geometric Properties}
\label{sec_geometry}

We start with some elementary and well-known properties of $\RLTweak$ and $\RLTstrong$.

\begin{proposition}
  \label{thm_rlt_elementary}
  Let $P \subseteq [0,1]^N$ be a polyhedron and let $B \subseteq N$.
  The following elementary properties hold for $\RLTweak_B(P)$ and $\RLTstrong_B(P)$.
  \begin{enumerate}[label=(\roman*)]
  \item
    \label{thm_rlt_elementary_redundant}
    $\RLTweak_B(P)$ and $\RLTstrong_B(P)$ is independent of the inequality description of $P$.
    In particular, scaling inequalities and adding redundant inequalities does not affect $\RLTweak_B(P)$ or $\RLTstrong_B(P)$.
  \item 
    \label{thm_rlt_elementary_complement}
    For any $k \in N$, the map $\sigma^k : \R^N \to \R^N$ that complements variable $x_k$, 
    \begin{equation*}
      \sigma^k(x)_i = \begin{cases}
        1 - x_i & \text{ if } i = k \\ x_i & \text{ if } i \neq k
        \end{cases}
    \end{equation*}
    satisfies $\RLTweak_B(\sigma^k(P)) = \sigma^k(\RLTweak_B(P))$ and $\RLTstrong_B(\sigma^k(P)) = \sigma^k(\RLTstrong_B(P))$.
  \item
    \label{thm_rlt_elementary_mccormick}
    Any $y \in \R^{B \times N}$ that fulfills~\eqref{eq_rlt_weak} also satisfies the \emph{McCormick inequalities}
    \begin{subequations}
      \label{eq_mccormick}
      \begin{alignat}{7}
        y_{i,j} &\leq x_i &\quad& \forall (i,j) \in B \times N, \label{eq_mccormick_i} \\
        y_{i,j} &\leq x_j &\quad& \forall (i,j) \in B \times N, \label{eq_mccormick_j} \\
        y_{i,j} &\geq x_i + x_j - 1 &\quad& \forall (i,j) \in B \times N, \label{eq_mccormick_ij} \\
        y_{i,j} &\geq 0 &\quad& \forall (i,j) \in B \times N, \label{eq_mccormick_nonneg}
      \end{alignat}
    \end{subequations}
    while any $\hat{y} \in \R^{N \times N}$ that fulfills~\eqref{eq_rlt_strong} even satisfies~\eqref{eq_mccormick} for all $(i,j) \in N \times N$.
  \item
    \label{thm_rlt_elementary_subset}
    $\conv\{ x \in \Z^B \times \R^{N \setminus B} \mid x \in P \} \subseteq \RLTstrong_B(P) \subseteq \RLTweak_B(P) \subseteq P$.
  \item
    \label{thm_rlt_elementary_monotone}
    Monotonicity: $\RLTweak_B(P) \subseteq \RLTweak_B(Q)$ and $\RLTstrong_B(P) \subseteq \RLTstrong_B(Q)$ hold for $P \subseteq Q$.
  \item
    \label{thm_rlt_elementary_faces}
    Taking faces and applying $\RLTweak_B(\cdot)$ or $\RLTstrong_B(\cdot)$ commute, that is,
    \begin{subequations}
      \begin{alignat}{7}
        \RLTweak_B( \{ x \in P \mid \transpose{a}x = \beta \} )
      &= \{ x \in \RLTweak_B(P) \mid \transpose{a}x = \beta \}  \\
        \RLTstrong_B( \{ x \in P \mid \transpose{a}x = \beta \} )
      &= \{ x \in \RLTstrong_B(P) \mid \transpose{a}x = \beta \}
      \end{alignat}
      \label{eq_rlt_elementary_faces}
    \end{subequations}
    holds for any inequality $\transpose{a}x \leq \beta$ that is valid for $P$.
  \end{enumerate}
\end{proposition}
\begin{proof}
  Throughout the proof we assume that $P = \{ x \in \R^N \mid Ax \leq b \}$ for some matrix $A \in \R^{M \times N}$ and $b \in \R^M$.
  We only prove the statements about $\RLTstrong_B(P)$ since those for $\RLTweak_B(P)$ are almost identical.

  \ref{thm_rlt_elementary_redundant}:
  Scaling of inequalities $Ax \leq b$ clearly results inequalities~\eqref{eq_rlt_first} and~\eqref{eq_rlt_second} that are scaled accordingly and does not affect the remaining inequalities, which means that $\RLTstrongEF_B(P)$ remains unchanged.
  Let $\transpose{a} x \leq \beta$ be a redundant inequality for $Ax \leq b$.
  By strong duality there exist $\lambda \in \Rnonneg^M$ such that $\transpose{\lambda}A = a$ and $\transpose{\lambda}b \leq \beta$.
  Applying these multipliers to~\eqref{eq_rlt_first} and~\eqref{eq_rlt_second} for any $j \in N$ yields
  $\transpose{\lambda}A y_{\star,j} \leq \transpose{\lambda} b x_j$ and $\transpose{\lambda}A (x - y_{\star,j}) \leq \transpose{\lambda} b x_j$, which in turn implies $\transpose{a} y_{\star,j} \leq \beta x_j$ and $\transpose{a} (x - y_{\star,j}) \leq \beta (1 - x_j)$.
  Hence, the inequalities~\eqref{eq_rlt_first} and~\eqref{eq_rlt_second} for $\transpose{a}x \leq \beta$ are redundant for $\RLTstrong_B(P)$.

  \ref{thm_rlt_elementary_complement}:
  Let $k \in N$.
  Since the map $\sigma^k$ is an involution, it suffices to show $\sigma^k(\RLTstrong_B(\sigma^k(P))) \subseteq \RLTstrong_B(P)$.
  The points $x' \in \sigma^k(P)$ are described by
  \begin{equation*}
    \sum_{i \neq k} A_{\star,i} x'_i - A_{\star,k} x'_k \leq b - A_{\star,k} .
  \end{equation*}
  Consequently, the points $(x',y') \in \RLTstrongEF_B(\sigma^k(P))$ are described by
  \begin{subequations}
    \label{eq_complemented_rlt}
    \begin{alignat}{7}
    \sum_{i \neq k} A_{\star,i} y'_{i,j} - A_{\star,k} y'_{k,j} &\leq (b - A_{\star,k}) x'_j &\quad& \forall j \in N \label{eq_complemented_rlt_first} \\
    \sum_{i \neq k} A_{\star,i} (x'_i - y'_{i,j}) - A_{\star,k} (x'_k - y'_{k,j}) &\leq (b - A_{\star,k}) (1 - x'_j) &\quad& \forall j \in N \label{eq_complemented_rlt_second} \\
    y'_{i,j} &= y'_{j,i} &\quad& \forall i,j \in N \label{eq_complemented_rlt_symmetry} \\
    y'_{j,j} &= x'_j &\quad& \forall j \in B \label{eq_complemented_rlt_diag}
    \end{alignat}
  \end{subequations}
  Now consider a point $x \in \sigma^k(\RLTstrong_B(\sigma^k(P)))$ and let $x' \coloneqq \sigma^k(x)$.
  Hence, there exist $(x',y') \in \R^{N \times N}$ that satisfy~\eqref{eq_complemented_rlt} as well as $x = \sigma^k(x')$.
  We consider the vector $y \in \R^{N \times N}$ defined by
  \[
    y_{i,j} \coloneqq \begin{cases}
      y'_{k,k} - 2x'_k - 1 & \text{if } i = k,~ j = k \\
      x'_i - y'_{i,k} & \text{if } i \neq k,~ j = k \\
      x'_j - y'_{k,j} & \text{if } i = k,~ j \neq k \\
      y'_{i,j} & \text{if } i \neq k,~ j \neq k
    \end{cases}
  \]
  for all $i,j \in N$, which clearly satisfies~\eqref{eq_rlt_symmetry} and~\eqref{eq_rlt_diag} for all $j \in B$ with $j \neq k$.
  If $k \in B$, then $y_{k,k} = y'_{k,k} - 2x'_k - 1 = - (1 - x_k) - 1 = x_k$ holds as well.
  Moreover,
  \begin{align*}
    Ay_{\star,k}
    &= \sum_{i \neq k} A_{\star,i} y_{i,k} + A_{\star,k} y_{k,k}
    = \sum_{i \neq k} A_{\star,i} (x'_i - y'_{i,k}) + A_{\star,k} (y'_{k,k} - 2x'_k + 1) \\
    &= \sum_{i \neq k} A_{\star,i} (x'_i - y'_{i,k}) - A_{\star,k}(x'_k - y'_{k,k})
       + A_{\star,k} (1 - x'_k) \\
    &\leq (b - A_{\star,k})(1 - x'_k) 
       + A_{\star,k} (1 - x'_k)
    = b ( 1 - x'_k)
    = b x_k
  \end{align*}
  holds, where the inequality follows from~\eqref{eq_complemented_rlt_second} for $j = k$.
  For each $j \in N \setminus \{k\}$ we have
  \begin{align*}
    Ay_{\star,j}
    &= \sum_{i \neq k} A_{\star,i} y_{i,j} + A_{\star,k} y_{k,j}
    = \sum_{i \neq k} A_{\star,i} y'_{i,j} + A_{\star,k} (x'_j - y'_{k,j}) \\
    &= \sum_{i \neq k} A_{\star,i} y'_{i,j} - A_{\star,k} y'_{k,j} + A_{\star,k} x'_j
    \leq (b - A_{\star,k}) x'_j + A_{\star,k} x'_j
    = b x'_j,
  \end{align*}
  where the inequality follows from~\eqref{eq_complemented_rlt_first}.
  In addition, we have
  \begin{align*}
    A(x - y_{\star,k})
    &= \sum_{i \neq k} A_{\star,i} (x_i - y_{i,k}) + A_{\star,k} (x_k - y_{k,k}) \\
    &= \sum_{i \neq k} A_{\star,i} (x_i - x'_i + y'_{i,k}) + A_{\star,k} (x_k - y'_{k,k} + 2x'_k - 1) \\
    &= \sum_{i \neq k} A_{\star,i} y'_{i,k} - A_{\star,k} y'_{k,k} + A_{\star,k} (x_k + 2x'_k - 1) \\
    &\leq (b - A_{\star,k}) x'_k + A_{\star,k} (x_k + 2x'_k - 1)
    = b (1 - x_k),
  \end{align*}
  where the inequality follows from~\eqref{eq_complemented_rlt_first} for $j = k$.
  For each $j \in N \setminus \{k\}$ we have
  \begin{align*}
    A(x - y_{\star,j})
    &= \sum_{i \neq k} A_{\star,i} (x_i - y_{i,j}) + A_{\star,k} (x_k - y_{k,j}) \\
    &= \sum_{i \neq k} A_{\star,i} (x'_i - y'_{i,j}) + A_{\star,k} (1 - x'_k - x'_j + y'_{k,j}) \\
    &= \sum_{i \neq k} A_{\star,i} (x'_i - y'_{i,j}) - A_{\star,k} (x'_k - y'_{k,j}) + A_{\star,k} (1 - x'_j) \\
    &\leq (b - A_{\star,k}) (1 - x'_j) + A_{\star,k} (1 - x'_j)
    = b ( 1 - x_j)
  \end{align*}
  where the inequality follows from~\eqref{eq_complemented_rlt_second}.
  Thus, $(x,y) \in \RLTstrongEF_B(P)$, which implies $x \in \RLTstrong_B(P)$ and concludes the proof.

  \ref{thm_rlt_elementary_mccormick}:
  We augment the system $Ax \leq b$ with the (redundant) inequalities $-x_i \leq 0$ and $x_i \leq 1$ for each $i \in N$.
  The corresponding (redundant) inequalities~\eqref{eq_rlt_first} are $-y_{i,j} \leq 0$ and $y_{i,j} \leq x_j$ (for each $j \in N$) and those for~\eqref{eq_rlt_second} are $-x_i + y_{i,j} \leq 0$ and $x_i - y_{i,j} \leq 1 - x_j$ (for each $j \in N$), which are equivalent to~\eqref{eq_mccormick}.

  \ref{thm_rlt_elementary_subset}:
  For the first containment, consider $x \in \Z^B \times \R^{N \setminus B}$ with $x \in P$.
  Define $y_{i,j} \coloneqq x_i \cdot x_j$ for all $i,j \in N$, which implies~\eqref{eq_rlt_symmetry} and~\eqref{eq_rlt_diag}.
  By multiplying $Ax \leq b$ with $x_j \geq 0$ we obtain~\eqref{eq_rlt_first}, and by multiplying it with $1 - x_j \geq 0$ we obtain~\eqref{eq_rlt_second}.
  This, together with convexity of $\RLTstrong_B(P)$, yields the first containment.
  The second containment follows from the fact that every solution $(x,y)$ to~\eqref{eq_rlt_strong} projects to a solution $(x,y_{\star,B})$ of~\eqref{eq_rlt_weak}.
  For the third one we add \eqref{eq_rlt_first} and~\eqref{eq_rlt_second} for any $j \in N$, which yields
  $Ax = A y_{\star,j} +  A (x - y_{\star,j} ) \leq b x_j + b (1-x_j) = b$.
  Hence, $Ax \leq b$ is valid for $\RLTstrong_B(P)$, which implies $\RLTstrong_B(P) \subseteq P$.

  \ref{thm_rlt_elementary_monotone} follows from \ref{thm_rlt_elementary_redundant} since $P \subseteq Q$ implies that the inequalities valid for $Q$, which are redundant for $P$, do not impose further restrictions on $\RLTstrong_B(P)$.

  \ref{thm_rlt_elementary_faces}:
  Let $F \coloneqq \{ x \in P \mid \transpose{a}x = \beta \}$ denote such a face of $P$.
  The ``$\subseteq$''-part of~\eqref{eq_rlt_elementary_faces} follows from~\ref{thm_rlt_elementary_monotone} applied to $F \subseteq P$ and from~\ref{thm_rlt_elementary_subset} applied to $F$.

  It remains to prove ``$\supseteq$''.
  Since $\transpose{a}x \leq \beta$ is valid for $P$, strong duality yields existence of $\lambda \in \Rnonneg^m$ such that $\transpose{\lambda}A = \transpose{a}$ and $\transpose{\lambda}b \leq \beta$.
  Now consider an $x \in \RLTstrong_B(P)$ that also satisfies $\transpose{a}x = \beta$.
  Let $y \in \R^{n \times n}$ be the corresponding vector such that~\eqref{eq_rlt_strong} holds.
  For each $j \in N$ we obtain
  \[
    \beta 
    = \transpose{a} x
    = \transpose{\lambda} A x
    = \transpose{\lambda} A y_{\star,j} + \transpose{\lambda} A (x - y_{\star,j})
    \leq \transpose{\lambda} b x_j + \transpose{\lambda} b (1 - x_j)
    = \transpose{\lambda} b
    \leq \beta.
  \]
  Hence, equality holds throughout, which implies $\transpose{a} y_{\star,j} = \beta x_j$ as well as $\transpose{a} (x - y_{\star,j}) = \beta x_j$.
  These are exactly the constraints of~\eqref{eq_rlt_strong} that hold for $\RLTstrong_B(F)$ in addition to those for $\RLTstrong_B(P)$, which establishes the claim.
  \qed
\end{proof}

\DeclareDocumentCommand\Bfrac{}{\ensuremath{B^{\textup{frac}}}}

Let $\Bfrac(x) \coloneqq \{ i \in B \mid 0 < x_i < 1 \} \subseteq B$ denote the subset of coordinates of binary variables for which a point $x \in \R^N$ is fractional.

\begin{theorem}
  \label{thm_rlt_characterization}
  Let $P \subseteq [0,1]^N$ be a polytope, let $B \subseteq N$, and consider a point $x \in P$.
  Then the statements below satisfy
  $
    \ref{enum_rlt_point} \Rightarrow \ref{enum_rlt_pointprime} \Leftrightarrow \ref{enum_rlt_lift_and_project} \Rightarrow \ref{enum_rlt_equations}
  $.
  Moreover, if $B = N$ then all statements are equivalent.
  \begin{enumerate}[label=(\roman*)]
    \item
    \label{enum_rlt_point}
    $x \in \RLTstrong_B(P)$ holds.
    \item 
    \label{enum_rlt_pointprime}    
    $x \in \RLTweak_B(P)$ holds.
  \item
    \label{enum_rlt_lift_and_project}
    There exist points $\bar{z}^{(i,\beta)} \in \{ z \in P \mid z_i = \beta \}$ for all $(i,\beta) \in \Bfrac(x) \times \{0,1\}$ such that
    \begin{subequations}
      \begin{alignat}{7}
        &x \in \conv\big\{ \bar{z}^{(i,0)}, \bar{z}^{(i,1)} \big\} &\quad& \forall i \in \Bfrac(x) \label{eq_rlt_lift_and_project} \\
        &\bar{z}^{(i,1)}_j \cdot x_i - \bar{z}^{(j,1)}_i \cdot x_j = 0 &\quad& \forall i,j \in \Bfrac(x) \label{eq_rlt_extra1_for2}
      \end{alignat}
    \end{subequations}
  \item
    \label{enum_rlt_equations}
    There exist points $\bar{z}^{(i,\beta)} \in \{ z \in P \mid z_i = \beta \}$ for all $(i,\beta) \in \Bfrac(x) \times \{0,1\}$ such that
    \begin{subequations}
      \begin{alignat}{7}
        \bar{z}^{(i,1)}_j \cdot x_i - \bar{z}^{(j,1)}_i \cdot x_j &= 0 &\quad& \forall i,j \in \Bfrac(x) \label{eq_rlt_extra1} \\
        \bar{z}^{(i,0)}_j \cdot (1-x_i) - (1-\bar{z}^{(j,1)}_i) \cdot x_j &= 0 &\quad& \forall i,j \in \Bfrac(x)  \label{eq_rlt_extra2} \\
        x_k &= \bar{z}^{(i,\beta)}_k &\quad& \forall i \in \Bfrac(x) ~ \forall \beta \in \{0,1\} \qquad \nonumber \\
        & &\quad& \qquad \quad \forall k \in B \setminus \Bfrac(x) \label{eq_rlt_fixed}
      \end{alignat}
    \end{subequations}
  \end{enumerate}
\end{theorem}

  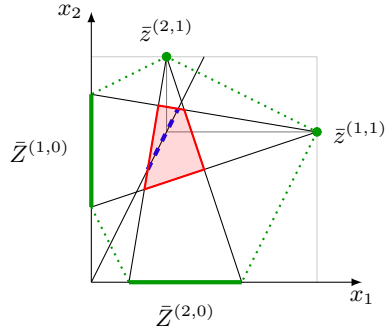
\begin{figure}[htpb]
    \centering
    \begin{tikzpicture}
    \draw[-latex] (0,0) -- (3.6,0) node[anchor=north] {$x_1$};
    \draw[-latex] (0,0) -- (0,3.6) node[anchor=east] {$x_2$};
    \draw[opacity = 0.2] (0,3) -- (3,3);
    \draw[opacity = 0.2] (3,0) -- (3,3);
    \draw[opacity = 0.5, thin] (1,2) -- (1,3);
    \draw[opacity = 0.5, thin] (1,2) -- (3,2);
    \draw (0,0) -- (1.5,3);
    \draw (1,3) -- (0.5,0) -- (2,0) -- cycle;
    \draw (3,2) -- (0,1) -- (0,2.5) -- cycle;
    \draw[dotted, thick, draw=green!60!black] (0,1) -- (0.5,0) -- (2,0) -- (3,2) -- (1,3) -- (0,2.5) -- (0,1) -- cycle;
    \draw[draw=green!60!black,ultra thick] (0,1) -- (0,2.5) node [midway] (unit3) {};
    \draw[draw=green!60!black,ultra thick] (0.5,0) -- (2,0) node [midway](unit4) {};
    \draw[draw=blue,ultra thick, dashed] (3/4,6/4) -- (15/13,30/13);
    \node[circle, draw=green!60!black, fill=green!60!black, inner sep=0.4mm] (unit1) at (3,2) {};
    \node[circle, draw=green!60!black, fill=green!60!black, inner sep=0.4mm] (unit2) at (1,3) {};
    \draw[draw=red, thick, fill=red, fill opacity=0.15] (3/2,3/2) -- (12/17,21/17) --  (33/37,87/37) -- (21/17,39/17) -- cycle;
    \node[right=1pt of unit1]{$\bar{z}^{(1,1)}$};
    \node[above=1pt of unit2] {$\bar{z}^{(2,1)}$};
    \node[left=1pt of unit3] {$\bar{Z}^{(1,0)}$};
    \node[below=1pt of unit4] {$\bar{Z}^{(2,0)}$};
    \end{tikzpicture}
    \caption{%
    Consider the polytope $P \subseteq [0,1]^2$ bounded by the \textcolor{green!60!black}{dotted lines} and the \textcolor{green!60!black}{lines} $\bar{Z}^{(1,0)}$ and $\bar{Z}^{(2,0)}$ for $B = \{1,2\}$ for which the points $\bar{z}^{(1,1)},~\bar{z}^{(2,1)}$ and all points $\bar{z}^{(i,0)} \in \bar{Z}^{(i,0)}$ for $i=1,2$ satisfy statement~\ref{enum_rlt_lift_and_project}. 
    Then \eqref{eq_rlt_lift_and_project} implies that all $x$ must lie in the \textcolor{red}{shaded region}, while \eqref{eq_rlt_extra1_for2} implies that all $x$ must be on the line through $\transpose{(0,0)}$ and $\transpose{(\bar{z}^{(2,1)}_1,\bar{z}^{(1,1)}_2)}$.
    Consequently, the \textcolor{blue}{dashed line segment} indicates all points $x \in P$ that satisfy \ref{enum_rlt_lift_and_project} for $z^{(i,1)}$ and $z^{(i,0)} \in \bar{Z}^{(i,0)}$ with $i=1,2$.
    \label{fig_rlt_z_relation}
  }
\end{figure}

\begin{figure}[htpb]
  \centering%
  \begin{tikzpicture}
    \draw[-latex] (0,0) -- (3.6,0) node[anchor=north] {$x_1$};
    \draw[-latex] (0,0) -- (0,3.6) node[anchor=east] {$x_2$};

    \draw[draw=red, thick, fill=red, fill opacity=0.15] (3,0) -- (0,0) -- (1.5,3) -- cycle;
    \draw[opacity=0.2] (0,3) -- (3,3);
    \draw[opacity=0.2] (3,0) -- (3,3);
    \draw[draw=blue, ultra thick, dashed] (0,0) -- (3,0);
    \node[circle, draw=black, fill=black, inner sep=0.4mm] (unit1) at (3,0) {};
    \node[circle, draw=black, fill=black, inner sep=0.4mm] (unit2) at (0,0) {};
    \node[circle, draw=black, fill=black, inner sep=0.4mm] (unit3) at (1.5,3) {};
    \node[below=1pt of unit1]{$(1,0)$};
    \node[below=1pt of unit2] {$(0,0)$};
    \node[above=1pt of unit3] {$(\frac{1}{2},1)$};
  \end{tikzpicture}
  \caption{%
    Consider the polytope $P \coloneqq \{ x \in [0,1]^2 \mid -2x_1 + x_2 \leq 0,~ 2x_1 + x_2 \leq 2 \}$ for $B = \{1\}$ (depicted as the \textcolor{red}{solid triangle}).
    By construction, the points $\bar{z}^{(1,\beta)}$ from \cref{thm_rlt_characterization}~\ref{enum_rlt_lift_and_project} and~\ref{enum_rlt_equations} are unique and satisfy $\bar{z}^{(1,0)} = \transpose{(0,0)}$ and $\bar{z}^{(1,1)} = \transpose{(1,0)}$.
    The point $x = \transpose{(\frac{1}{2},1)}$, for example, together with $\bar{z}$ satisfies statement~\ref{enum_rlt_equations} but violates statement~\ref{enum_rlt_lift_and_project}.
    All points satisfying statement~\ref{enum_rlt_lift_and_project}, that is $\RLTweak_B(P)$, are depicted as the \textcolor{blue}{dashed line}, while all points in the shaded region satisfy statement~\ref{enum_rlt_equations}.
    \label{fig_example_4_not_3}
  }
\end{figure}
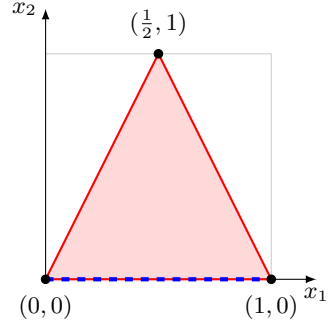

\begin{proof}
  Throughout the proof we assume that $P = \{ x \in \R^N \mid Ax \leq b \}$ for some matrix $A \in \R^{M \times N}$ and $b \in \R^M$.
  By definition of $\RLTstrong_B(P)$ and $\RLTweak_B(P)$, \ref{enum_rlt_point} implies \ref{enum_rlt_pointprime}.
  
  \medskip
  
  Now we will show that \ref{enum_rlt_pointprime} implies \ref{enum_rlt_lift_and_project}.
  By \ref{enum_rlt_pointprime} there exists a $y \in \R^{N \times B}$ such that~\eqref{eq_rlt_weak} holds.
  For each $i \in \Bfrac(x)$, let $\bar{z}^{(i,1)} \coloneqq \frac{1}{x_i} y_{\star,i}$ and $\bar{z}^{(i,0)} \coloneqq \frac{1}{1-x_i} (x - y_{\star,i})$.
  From \eqref{eq_rlt_weak_first} we get $\bar{z}^{(i,1)} \in P$ and \eqref{eq_rlt_weak_diag} gives $\bar{z}^{(i,1)}_i = 1$.
  Similarly, \eqref{eq_rlt_weak_second} implies $\bar{z}^{(i,0)} \in P$ and~\eqref{eq_rlt_weak_diag} yields $\bar{z}^{(i,0)}_i = 0$.
  Let $i,j \in \Bfrac(x)$.
  We establish~\eqref{eq_rlt_extra1_for2} via
  \begin{equation*}
    \bar{z}^{(i,1)}_j \cdot x_i - \bar{z}^{(j,1)}_i \cdot x_j
    = y_{j,i} \frac{x_i}{x_i} - y_{i,j} \frac{x_j}{x_j}
    = y_{j,i} - y_{i,j} = 0,
  \end{equation*}
  where the last equality follows from~\eqref{eq_rlt_symmetry}.
  To show~\eqref{eq_rlt_lift_and_project}, note that $x_i \in [0,1]$ for $i\in \Bfrac(x)$ and
  $
    x_i \cdot \bar{z}^{(i,1)} + (1-x_i) \cdot \bar{z}^{(i,0)} = y_{\star,i} - (x-y_{\star,1}) = x
  $
  holds.

  \medskip

  Third, we will show that \ref{enum_rlt_lift_and_project} implies \ref{enum_rlt_pointprime}.
  For all $i \in \Bfrac(x)$ and $\beta = 0,1$, let $\bar{z}^{(i,\beta)}$ be as in \ref{enum_rlt_lift_and_project}.
  We define $y \in \R^{N \times B}$ via
  \[
    y_{i,j} \coloneqq \begin{cases}
      x_j \cdot \bar{z}^{(j,1)}_i & \text{ if } j \in \Bfrac(x)  \\
      x_i \cdot \bar{z}^{(i,1)}_j & \text{ if } i \in \Bfrac(x),~j \notin \Bfrac(x)  \\
      x_i \cdot x_j &\text{ otherwise}
    \end{cases}
    \qquad \forall i\in N, j \in B.
  \]
  Note that~\eqref{eq_rlt_extra1_for2} implies $y_{i,j} = y_{j,i}$ for $i,j \in \Bfrac(x)$, that is, \eqref{eq_rlt_weak_symmetry} holds.
  For all other $i,j$, it follows by construction.
  For $k \in B \setminus \Bfrac(x)$ we have $y_{k,k} = x_k^2 = x_k$.
  For $j \in \Bfrac(x)$, we have $y_{j,j} = x_j \cdot \bar{z}^{(j,1)}_j = x_j$.
  Thus, \eqref{eq_rlt_weak_diag} follows.

\noindent
  For each $j \in \Bfrac(x)$ we obtain $A y_{\star,j} = A \bar{z}^{(j,1)} x_j \leq b x_j$ since $\bar{z}^{(j,1)} \in P$, that is, \eqref{eq_rlt_weak_first} holds for all $j \in \Bfrac(x)$.
  Since $\bar{z}_j^{(j,0)} = 0$ and $\bar{z}_j^{(j,1)} = 1$ holds, every convex combination $x = \lambda \cdot \bar{z}^{(j,1)} + (1-\lambda) \cdot \bar{z}^{(j,0)}$ has $x_j = \lambda$. Therefore,
  \begin{equation*}
    A (x-y_{\star,j})
    = A(x_j \cdot \bar{z}^{(j,1)} + (1-x_j) \cdot \bar{z}^{(j,0)} -x_j \cdot \bar{z}^{(j,1)})
    = A\bar{z}^{(j,0)}(1-x_j)
    = b(1-x_j).
  \end{equation*}
  Thus, \eqref{eq_rlt_weak_second} also holds for all $j \in \Bfrac(x)$.
   Suppose $x_k = 0$ for $k \in B \setminus \Bfrac(x)$. By assumption $z_k^{(i,1)} = 0$ holds for all $i\in \Bfrac(x)$.
  Consequently, $y_{i,k} = 0$ holds for all $i \in N$, and~\eqref{eq_rlt_weak_first} follows immediately.
  From $x \in P$ we have $A x \leq b$ and thus $A (x - y_{\star,k}) = A x \leq b = b (1 - 0)$, that is, \eqref{eq_rlt_weak_second} holds.

  Suppose $x_k = 1$. Then, for $i\in \Bfrac(x)$ we have $\bar{z}_k^{(i,1)} = 1$ by assumption.
  Therefore, $x_i - y_{i,k} = 0$ holds for all $i \in N$, and~\eqref{eq_rlt_weak_second} follows immediately.
  From $x \in P$ we have $A y_{\star,k} = A x \leq b = b x_k$, that is, \eqref{eq_rlt_weak_first} holds.
  As~\eqref{eq_rlt_weak} holds for $(x,y)$, we conclude with $x \in \RLTweak_B(P)$.
  
\medskip

  We now show that \ref{enum_rlt_lift_and_project} implies \ref{enum_rlt_equations}.
  Again, as $\bar{z}_i^{(i,0)} = 0$ and $\bar{z}_i^{(i,1)} = 1$ holds, every convex combination $x = \lambda \cdot \bar{z}^{(i,1)} + (1-\lambda) \cdot \bar{z}^{(i,0)}$ has $x_i = \lambda$.
  Hence, \eqref{eq_rlt_lift_and_project} implies $x_j = x_i \cdot \bar{z}^{(i,1)}_j + (1 - x_i) \cdot \bar{z}_j^{(i,0)}$ for $i \in \Bfrac(x)$ and $j \in N$.
  Therefore, for $i,j \in \Bfrac(x)$,
  \begin{align*}
      \bar{z}^{(i,0)}_j \cdot (1-x_i) - (1-\bar{z}^{(j,1)}_i) \cdot x_j & = \left (x_j - x_i \cdot \bar{z}_j^{(i,1)} \right ) - (1-\bar{z}^{(j,1)}_i)\cdot x_j \\
      & = x_j \cdot \bar{z}^{(j,1)}_i - x_i \cdot \bar{z}_j^{(i,1)}=0,
  \end{align*}
  where the last equality follows from \eqref{eq_rlt_extra1_for2}. To see why equation~\eqref{eq_rlt_fixed} is satisfied, consider $j \in \Bfrac(x)$ and $k \in B \setminus \Bfrac(x)$.
  Suppose $x_k = 1$ holds. 
  Note that because $\bar{z}^{(j,\beta)} \in P$, we have $\bar{z}^{(j,\beta)}_k \leq 1$.
  We obtain \begin{equation*}
      1 = x_k = x_j \cdot \bar{z}^{(j,1)}_k + (1 - x_j) \cdot \bar{z}_k^{(j,0)} \leq x_j \cdot 1 + (1-x_j) \cdot 1 = 1
  \end{equation*}
  Equality must hold and therefore $\bar{z}_k^{(j,0)} = \bar{z}_k^{(j,1)} = 1 = x_k$ holds. Similarly, as $\bar{z}^{(j,\beta)} \in P$, we have $\bar{z}^{(j,\beta)}_k \geq 0$ and \begin{equation*}
      0 = x_k = x_j \cdot \bar{z}^{(j,1)}_k + (1 - x_j) \cdot \bar{z}_k^{(j,0)} \geq x_j \cdot 0 + (1-x_j) \cdot 0 = 0
  \end{equation*} 
  Consequently $\bar{z}_k^{(j,0)} = \bar{z}_k^{(j,1)} = 0 = x_k$.
  
\medskip

  We finally show that \ref{enum_rlt_equations} implies \ref{enum_rlt_lift_and_project} if $B = N$.
  For all $i \in \Bfrac(x)$ and $\beta = 0,1$, let $\bar{z}^{(i,\beta)}$ be as in \ref{enum_rlt_equations}.
  Since~\eqref{eq_rlt_extra1_for2} is equivalent to~\eqref{eq_rlt_extra1},
  we only need to show~\eqref{eq_rlt_lift_and_project}.
  To this end, we add~\eqref{eq_rlt_extra1} and~\eqref{eq_rlt_extra2}, which yields
  \begin{multline*}
    0
    = \bar{z}^{(i,1)}_j \cdot x_i - \bar{z}^{(j,1)}_i \cdot x_j
    + \bar{z}^{(i,0)}_j \cdot (1 - x_i) - (1-\bar{z}^{(j,1)}_i) \cdot x_j  \\
    = \bar{z}^{(i,1)}_j \cdot x_i + \bar{z}^{(i,0)}_j \cdot (1 - x_i) - x_j.
  \end{multline*}
  We obtain $x_j = \bar{z}^{(i,1)}_j \cdot x_i + \bar{z}^{(i,0)}_j \cdot (1 - x_i)$ for each $j \in \Bfrac(x)$.
  For each $k \in B \setminus \Bfrac(x)$, equation~\eqref{eq_rlt_fixed} implies $x_k = \bar{z}^{(i,1)}_k = \bar{z}^{(i,0)}_k$ and thus also $x_k = \bar{z}^{(i,1)}_k \cdot x_i + \bar{z}^{(i,0)}_k \cdot (1 - x_i)$.
  Due to $B = N$, this covers all $j \in N$, and we conclude that $x = \bar{z}^{(i,1)} \cdot x_i + \bar{z}^{(i,0)} \cdot (1 - x_i)$ holds, which implies~\eqref{eq_rlt_lift_and_project} due to $x_i \in [0,1]$.
  \qed
\end{proof}

\Cref{fig_rlt_z_relation} visualizes the impact of conditions~\eqref{eq_rlt_lift_and_project} and~\eqref{eq_rlt_extra1_for2}.
\Cref{fig_example_4_not_3} depicts an example showing that \ref{enum_rlt_equations} does not imply \ref{enum_rlt_lift_and_project}.

\begin{remark}
\label{rem_rlt_calc_z}
    For a point $(x,y) \in \RLTstrongEF_B(P)$ we can calculate the corresponding $\bar{z}^{(i,\beta)}$ as
    $\bar{z}^{(i,1)} = \frac{1}{x_i} y_{\star,i}$ and $\bar{z}^{(i,0)} = \frac{1}{1-x_i} (x - y_{\star,i})$ for $i \in \Bfrac(x)$.
\end{remark}

For the pure integer case $B = N$, \cref{thm_rlt_characterization} shows that the RLT closures only depend on the intersection $P$ with the faces of the unit cube.

\begin{corollary}
 \label{thm_rlt_boundary}
  $\RLTstrong_N(P) = \RLTweak_N(P)$ only depends the intersection of $P$ with the boundary of the unit cube $[0,1]^N$.
\end{corollary}

\section{Quadratic assignment problem}
\label{sec_qap}

\DeclareDocumentCommand\Xassign{}{\ensuremath{\mathcal{X}}}
\DeclareDocumentCommand\Passign{}{\ensuremath{\mathcal{P}}}

In this section we consider the \emph{quadratic assignment problem} in its general form
\begin{multline}
  \text{min } \sum_{i,j,k,\ell \in N} q_{i,j,k,\ell} x_{i,j} x_{k,\ell} \\ 
  \text{s.t. } \sum_{i \in N} x_{i,j} = 1 ~ \forall j \in N,~ \sum_{j \in N} x_{i,j} = 1 ~\forall i \in N,~ x \in \{0,1\}^{N \times N},
  \label{model_qap}
\end{multline}
where $N$ is a (finite) ground set and where $q \in \Rnonneg^{N \times N \times N \times N}$ are costs.
We denote its set of feasible solutions by $\Xassign_N \subseteq \{0,1\}^{N \times N}$, and by $\Passign_N \coloneqq \conv(\Xassign_N)$ its convex hull.
It is well-known that $\Passign_N$ is described the equations of~\eqref{model_qap} and nonnegativity~\cite{Birkhoff46,Neumann53}.
We consider different ways of reformulating~\eqref{model_qap} as a mixed-integer program.

The first one is generic and covers several formulations with additional variables $w \in \R^{N \times N}$ with generic coupling inequalities
\begin{equation}
  \onevec w_{i,j} \geq A^{(i,j)}x + b^{(i,j)} \quad \forall i,j \in N \label{eq_qap_x_w}
\end{equation}
(with $\onevec$ being the all-$1$s vector compatible with the dimensions of $A^{(i,j)}$ and $b^{(i,j)}$) that model
\begin{equation}
  w_{i,j} = x_{i,j} \sum_{k,\ell \in N} q_{i,j,k,\ell} x_{k,\ell}   
  \label{eq_qap_coupling_implication}
\end{equation}
whenever $x_{i,j}$ is binary.
Clearly, $A^{(i,j)}$ and $b^{(i,j)}$ may depend on the costs $q$.
Note that larger $w$-values are feasible for~\eqref{eq_qap_x_w}, but~\eqref{eq_qap_coupling_implication} will hold for every optimal solution when minimizing $\sum_{i,j \in N} w_{i,j}$.
The formulation then reads
\begin{equation}
  \text{min } \sum_{i,j \in N} w_{i,j} \text{ s.t. } x \in \Xassign_N,~ (w,x) \text{ satisfies~\eqref{eq_qap_x_w}}.
  \label{model_qap_compact}
\end{equation}
Several formulations from the literature follow that pattern, potentially after applying an affine transformation.
Worth mentioning are the Kaufman-Broeckx linearization~\cite{KaufmanB78}, the Xian-Yuan linearization~\cite{XiaY06} and the Gilmore-Lawler linearization proposed in~\cite{ZhangBM13}.

The second formulation is essentially $\RLTweakEF_{(N \times N)}(\Passign_N)$ and is due to Adams and Johnson~\cite{AdamsJ94}.
For our purposes we augment it by similar $w$-variables in a way that does not affect the strength:
\begin{subequations}
  \label{model_qap_aj}
  \begin{alignat}{7}
    & \text{min } ~ \mathrlap{ \sum_{i,j \in N} w_{i,j} } \\
    & \text{s.t. }
      & x &\in \Passign_N \label{model_qap_aj_assignment} \\
    & & \sum_{i \in N} y_{i,j,k,\ell} &= x_{k,\ell} &\quad& \forall j,k,\ell \in N \label{model_qap_aj_first} \\
    & & \sum_{j \in N} y_{i,j,k,\ell} &= x_{k,\ell} &\quad& \forall i,k,\ell \in N \label{model_qap_aj_second} \\
    & & y_{i,j,i,j} &= x_{i,j} &\quad& \forall i,j \in N \label{model_qap_aj_diag} \\
    & & y_{i,j,k,\ell} = y_{k,\ell,i,j} &\geq 0 &\quad& \forall i,j,k,\ell \in N \label{model_qap_aj_symmetry_nonneg} \\
    & & \sum_{k,\ell \in N} q_{i,j,k,\ell} y_{i,j,k,\ell} &= w_{i,j}  &\quad& \forall i,j \in N \label{model_qap_aj_w}
  \end{alignat}
\end{subequations}
Our first result in this section establishes dominance of the first class of formulations regarding the strength of their LP relaxations.

\begin{theorem}
  \label{thmQAPdominanceSingle}
  The LP bound of~\eqref{model_qap_aj} is greater than or equal to that of any formulation of the form~\eqref{model_qap_compact}.
\end{theorem}

The dominance over the Xia-Yuan linearization~\cite{XiaY06} was already established in~\cite{HuberR17}.
Since RLT dominates lift-and-project, this result may be not surprising since~\eqref{eq_qap_coupling_implication} becomes linear in $x$ if only a single variable $x_{i,j}$ is fixed.

In fact we can prove a stronger dominance relation, which yields \cref{thmQAPdominanceSingle} as a consequence.
For this we consider a third class of formulations that may involve two types of coupling constraints
\begin{subequations}
  \label{eq_qap_x_w_cols_rows}
  \begin{alignat}{7}
    \sum_{i \in N} a^{(i,j)} w_{i,j} &\geq B^{(j)} x + c^{(j)} &\quad& \forall j \in N \label{eq_qap_x_w_cols} \\
    \sum_{j \in N} d^{(i,j)} w_{i,j} &\geq E^{(i)} x + f^{(i)} &\quad& \forall i \in N, \label{eq_qap_x_w_rows}
  \end{alignat}
\end{subequations}
where~\eqref{eq_qap_x_w_cols} for a particular $j \in N$ implies~\eqref{eq_qap_coupling_implication} for all $i \in N$ simultaneously (assuming $x_{\star,j}$ is binary), and where~\eqref{eq_qap_x_w_rows} for a particular $i \in N$ implies~\eqref{eq_qap_coupling_implication} for all $j \in N$ simultaneously (assuming $x_{i,\star}$ is binary).
Such formulations can be stronger than~\eqref{model_qap_compact} since it permits inequalities involving more than one $w$-variable.
More precisely, it subsumes all formulations with $x$- and $w$-variables such that each constraint only involves $w_{i,j}$-variables for at most one $i$ or for at most one $j$.
Such a formulation reads
\begin{equation}
  \text{min } \sum_{i,j \in N} w_{i,j} \text{ s.t. } x \in \Xassign_N,~ (w,x) \text{ satisfies~\eqref{eq_qap_x_w_cols_rows}}.
  \label{model_qap_cols_rows}
\end{equation}

\begin{theorem}
  \label{thmQAPdominanceColsRows}
  The LP bound of~\eqref{model_qap_aj} is greater than or equal to that of any formulation of the form~\eqref{model_qap_cols_rows}.
\end{theorem}

\begin{proof}
  Let $(x^\star,y^\star,w^\star) \in \R^{N \times N} \times \R^{N \times N \times N \times N} \times \R^{N \times N}$ be a feasible relaxation solution to~\eqref{model_qap_aj}.
  By~\eqref{model_qap_aj_assignment}, $x^\star \in \Passign_N$, so it remains to show~\eqref{eq_qap_x_w_cols_rows}.
  By symmetry of the formulation with respect to exchanging $i$ and $j$ (as well as $k$ and $\ell$), it suffices to show that~\eqref{eq_qap_x_w_cols} holds.
  To this end, we consider any $j \in N$.
  For each $i \in I^+ \coloneqq \{ i \in N \mid x^\star_{i,j} > 0 \}$ we construct vectors $\bar{x}^i \in \R^{N \times N}$ and $\bar{v}^i \in \R^N$ via
  \begin{alignat*}{7}
    \bar{x}^i_{k,\ell} &\coloneqq \frac{ 1 }{ x^\star_{i,j} } y^\star_{i,j,k,\ell} &\quad& \forall k,\ell \in N \\
    \bar{v}^i_i &\coloneqq w^\star_{i,j} / x^\star_{i,j} \\ 
    \bar{v}^i_{i'} &\coloneqq 0 &\quad& \forall i' \in N \setminus \{i\}
  \end{alignat*}
  They satisfy
  \begin{alignat*}{7}
    \sum_{k \in N} \bar{x}^i_{k,\ell} &= \sum_{k \in N} y^\star_{i,j,k,\ell} / x^\star_{i,j} \overset{\eqref{model_qap_aj_first}}{=} 1 \quad&& \forall \ell \in N, \\
    \sum_{\ell \in N} \bar{x}^i_{k,\ell} &= \sum_{\ell \in N} y^\star_{i,j,k,\ell} / x^\star_{i,j} \overset{\eqref{model_qap_aj_second}}{=} 1 \quad&& \forall k \in N, \\
    \bar{x}^i_{i,j} &= \frac{ y^\star_{i,j,i,j} }{ x^\star_{i,j} } \overset{ \eqref{model_qap_aj_diag} }{ = } 1 \\
    \bar{x}^i_{k,\ell} &= y^\star_{i,j,k,\ell} / x^\star_{i,j} \overset{\eqref{model_qap_aj_symmetry_nonneg}}{\geq} 0  \quad&& \forall k,\ell \in N, \text{ and} \\
    \bar{v}^i_i &= \frac{ w^\star_{i,j} }{ x^\star_{i,j} }
    \overset{ \eqref{model_qap_aj_w} }{ = } \frac{ 1 }{ x^\star_{i,j} } \sum_{k,\ell \in N} q_{i,j,k,\ell} y^\star_{i,j,k,\ell}
    =&& \sum_{k,\ell \in N} q_{i,j,k,\ell} \bar{x}^i_{k,\ell}.
  \end{alignat*}
  Hence, $\bar{x}^i \in \Passign$ with $\bar{x}^i_{i,j} = 1$ and consequently $\bar{x}^i_{k,j} = 0$ for all $k \in N \setminus \{i\}$.
  Thus, $(\bar{x}^i, \bar{v}^i)$ satisfies
  \begin{equation*}
    \bar{v}^i_{i'} = \bar{x}^i_{i',j} \sum_{k,\ell \in N} q_{i',j,k,\ell} \bar{x}^i_{k,\ell},
  \end{equation*}
  for all $i' \in N$, which corresponds to~\eqref{eq_qap_coupling_implication} for $w_{\star,j} \coloneqq \bar{v}^i$.
  Any such vector $(w(\bar{v}^i),\bar{x}^i)$ will thus satisfy~\eqref{eq_qap_x_w_cols} since these inequalities only depend on $w_{k,\ell}$ for $\ell = j$).
  Moreover, note that
  \begin{equation*}
    x^\star_{k,\ell} 
    \overset{ \eqref{model_qap_aj_first} }{=} \sum_{ i \in N } y^\star_{i,j,k,\ell}
    = \sum_{ i \in I^+ } y^\star_{i,j,k,\ell}
    = \sum_{ i \in I^+ } x^\star_{i,j} \cdot \bar{x}^i_{k,\ell}
  \end{equation*}
  holds for all $k,\ell \in N$, where the second equation holds since $i \in N \setminus I^+$ means $x^\star_{i,j} = 0$, which implies $y^\star_{i,j,k,\ell} = 0$.
  Finally, for $i \in N$ we have $w^\star_{i,j} = x^\star_{i,j} \cdot \bar{v}^i$.
  Together, this shows that $(w^\star_{\star,j}, x^\star)$ is a convex combination of the vectors $(\bar{v}^i,\bar{x}^i)$ for all $i \in I^+$.
  Consequently, also $(w^\star,x^\star)$ satisfies the constraints~\eqref{eq_qap_x_w_cols}, which concludes the proof.
  \qed
\end{proof}
In the next section we investigate in a general setting which disjunctive hulls are dominated by an RLT formulation.

\section{Disjunctive hulls dominated by RLT}
\label{sec_dominance}

\Cref{thmQAPdominanceColsRows} exhibits an example in which an RLT closure dominates disjunctive hulls for disjunctions that are more complicated than a variable disjunction.
In this section we investigate which types of disjunctions are implied by (which) RLT closures.
In fact we characterize disjunctions whose disjunctive hull is dominated by $\RLTstrong(\cdot)$, while we provide necessary conditions for the case of $\RLTweak(\cdot)$.

\begin{theorem}
  \label{thm_cardinality_equation}
  Let $P \subseteq [0,1]^N$ be a polyhedron for which $\sum_{j \in J} x_j = 1$ holds for some index set $J \subseteq B$.
  Then
  \begin{equation*}
    \RLTstrong_B(P) \subseteq \conv \Big( \bigcup_{j \in J} \{ x \in P \mid x_j = 1 \} \Big).
  \end{equation*}
\end{theorem}

\begin{proof}
  We assume that $P = \{ x \in \R^N \mid Ax \leq b \}$ for some matrix $A \in \R^{M \times N}$ and $b \in \R^M$.
  Let $x \in \RLTstrong_B(P)$ and let $y \in \R^{N \times N}$ be a corresponding vector such that~\eqref{eq_rlt_strong} holds.
  Let $J^+ \coloneqq \{ j \in J \mid x_j > 0 \}$ index those variables of the equation that are positive.
  For each $j \in J^+$ we define a vector $z^j \coloneqq 1/x_j \cdot y_{\star,j} \in \R^N$.
  From~\eqref{eq_rlt_first} we obtain $Az^j = 1/x_j \cdot Ay_{\star,j} \leq 1/x_j \cdot b x_j = b$, that is, $z^j \in P$ holds.
  Moreover, $z^j_j = y_{j,j} / x_j = 1$ holds due to~\eqref{eq_rlt_diag}.
  Hence, it remains to show that $x$ is a convex combination of the points $z^j$ for $j \in J^+$.
  For each $i \in N$ we observe
  \begin{equation}
    \sum_{j \in J^+} x_j \cdot z^j_i
    = \sum_{j \in J^+} y_{i,j}
    = \sum_{j \in J} y_{i,j}
    = \sum_{j \in J} y_{j,i}
    = 1 \cdot x_i,
    \label{eq_cardinality_equation_convex}
  \end{equation}
  where the second equality follows from~\eqref{eq_rlt_symmetry}, the third from the fact that $0 \leq y_{i,j} \leq x_j = 0$ follows from \cref{thm_rlt_elementary}~\ref{thm_rlt_elementary_mccormick} for each $j \in J \setminus J^+$ and the last one from the fact that $\sum_{j \in J} y_{j,i} = x_i$ is valid for $(x,y)$, which can be derived from $\sum_{j \in J} x_j = 1$ by \cref{thm_rlt_elementary}~\ref{thm_rlt_elementary_redundant}.
  \qed
\end{proof}

Note that the very last step in the proof of \cref{thm_cardinality_equation} does not apply if we replace $\RLTstrong_B(P)$ by $\RLTweak_B(P)$.
Unfortunately, we are not aware of an example showing that a similar theorem does not hold for $\RLTweak_B(P)$.
However, the example in \cref{fig_landp_does_not_dominate} shows that lift-and-project does \emph{not} dominate such disjunctive hulls.

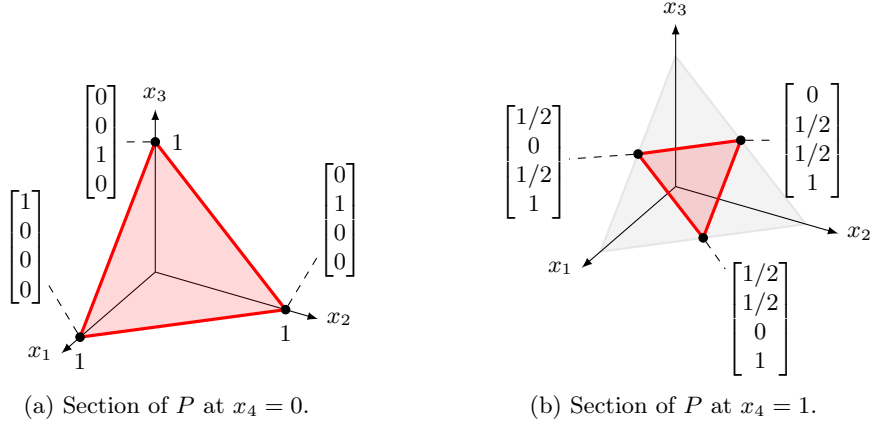
\begin{figure}[htpb]
  
  \subfloat[Section of $P$ at $x_4 = 0$.]{%
    \label{fig_landp_does_not_dominate_first}
    \centering
    \tdplotsetmaincoords{60}{120}
    \begin{tikzpicture}[
      tdplot_main_coords,
    ]
    \draw[-latex] (0,0,0) -- (2.5,0,0) node[anchor=east] {$x_1$};
    \draw[-latex] (0,0,0) -- (0,2.5,0) node[anchor=west] {$x_2$};
    \draw[-latex] (0,0,0) -- (0,0,2.5) node[anchor=south] {$x_3$};

    \draw[draw=red, very thick, fill=red, fill opacity=0.15] (2,0,0) -- (0,2,0) -- (0,0,2) -- cycle;

    \node[circle, draw=black, fill=black, inner sep=0.4mm] (unit1) at (2,0,0) {};
    \node[above left=15pt of unit1, inner sep=0] (unit1label) {$\begin{bmatrix} 1 \\ 0 \\ 0 \\ 0 \end{bmatrix}$};
    \draw[dashed] (unit1) -- (unit1label);

    \node[circle, draw=black, fill=black, inner sep=0.4mm] (unit2) at (0,2,0) {};
    \node[above right=15pt of unit2, inner sep=0] (unit2label) {$\begin{bmatrix} 0 \\ 1 \\ 0 \\ 0 \end{bmatrix}$};
    \draw[dashed] (unit2) -- (unit2label);

    \node[circle, draw=black, fill=black, inner sep=0.4mm] (unit3) at (0,0,2) {};
    \node[left=10pt of unit3, inner sep=0] (unit3label) {$\begin{bmatrix} 0 \\ 0 \\ 1 \\ 0 \end{bmatrix}$};
    \draw[dashed] (unit3) -- (unit3label);

    \node[below=1pt of unit1] {$1$};
    \node[below=1pt of unit2] {$1$};
    \node[right=1pt of unit3] {$1$};

    \end{tikzpicture}
  }%
  \hfill
  \subfloat[Section of $P$ at $x_4 = 1$.]{%
    \label{fig_landp_does_not_dominate_second}
    \centering
    \tdplotsetmaincoords{60}{120}
    \begin{tikzpicture}[
      tdplot_main_coords,
    ]
    \draw[-latex] (0,0,0) -- (2.5,0,0) node[anchor=east] {$x_1$};
    \draw[-latex] (0,0,0) -- (0,2.5,0) node[anchor=west] {$x_2$};
    \draw[-latex] (0,0,0) -- (0,0,2.5) node[anchor=south] {$x_3$};

    \draw[draw=black, thick, fill=black, draw opacity=0.1, fill opacity=0.05] (2,0,0) -- (0,2,0) -- (0,0,2) -- cycle;
    \draw[draw=red, very thick, fill=red, fill opacity=0.15] (1,1,0) -- (0,1,1) -- (1,0,1) -- cycle;

    \node[circle, draw=black, fill=black, inner sep=0.4mm] (vertex1) at (1,0,1) {};
    \node[above left=15pt of unit1, inner sep=0] (vertex1label) {$\begin{bmatrix} 1/2 \\ 0 \\ 1/2 \\ 1 \end{bmatrix}$};
    \draw[dashed] (vertex1) -- (vertex1label);

    \node[circle, draw=black, fill=black, inner sep=0.4mm] (vertex2) at (0,1,1) {};
    \node[right=12pt of vertex2, inner sep=0] (vertex2label) {$\begin{bmatrix} 0 \\ 1/2 \\ 1/2 \\ 1 \end{bmatrix}$};
    \draw[dashed] (vertex2) -- (vertex2label);

    \node[circle, draw=black, fill=black, inner sep=0.4mm] (vertex3) at (1,1,0) {};
    \node[below right=10pt of vertex3, inner sep=0] (vertex3label) {$\begin{bmatrix} 1/2 \\ 1/2 \\ 0 \\ 1 \end{bmatrix}$};
    \draw[dashed] (vertex3) -- (vertex3label);

    \end{tikzpicture}
  }%
  \caption{%
    Consider the polytope $P \subseteq \R^N$ with $N = \{1,2,3,4\}$ defined as the convex hull of the three vertices in \cref{fig_landp_does_not_dominate_first} and the three vertices in \cref{fig_landp_does_not_dominate_second}.
    On the one hand, the equation $x_1 + x_2 + x_3 = 1$ is valid for $P$, and hence \cref{thm_cardinality_equation} shows that $\RLTstrong_N(P)$ is (contained in and thus) equal to the disjunctive hull for the disjunction $x_1 = 1 \lor x_2 = 1 \lor x_3 = 1$, which is the triangle in \cref{fig_landp_does_not_dominate_first}.
    On the other hand, the point $x^\star = \transpose{ \big( \frac{1}{3}, \frac{1}{3}, \frac{1}{3}, \frac{2}{3} \big) }$ belongs to $\LandP_N(P)$.
  }
  \label{fig_landp_does_not_dominate}
\end{figure}

\DeclareDocumentCommand\red{m}{\textcolor{red}{#1}}
\DeclareDocumentCommand\blue{m}{\textcolor{blue}{#1}}

\begin{figure}[htpb]
  \subfloat[%
    Section of $P$ at $x_1 = 0$.
    \label{fig_rlt_does_not_dominate_ineq_picture}
  ]{%
    \centering
    \tdplotsetmaincoords{80}{17}
    \begin{tikzpicture}[
      tdplot_main_coords,
    ]
    \draw[-latex] (0,0,0) -- (2.5,0,0) node[anchor=north] {$x_2$};
    \draw[-latex] (0,0,0) -- (0,5.0,0) node[anchor=west] {$x_3$};
    \draw[-latex] (0,0,0) -- (0,0,2.5) node[anchor=west] {$x_4$};

    \draw[draw=black, very thick] (1,0,0) -- (2,0,0) -- (0,0,2) -- (0,2,0) -- (0,1,0);
    \draw[draw=red, very thick, fill=red, fill opacity=0.15] (2,0,0) -- (0,2,0) -- (0,0,2) -- cycle;
    \draw[draw=blue, very thick, fill=blue, fill opacity=0.15] (1,0,0) -- (0,1,0) -- (0,0,2) -- cycle;

    \node[circle, draw=black, fill=black, inner sep=0.4mm] (unit1) at (2,0,0) {};
    \node[above right=15pt of unit1, inner sep=0] (unit1label) {$\begin{bmatrix} 0 \\ 1 \\ 0 \\ 0 \end{bmatrix}$};
    \draw[dashed] (unit1) -- (unit1label);

    \node[circle, draw=black, fill=black, inner sep=0.4mm] (unit2) at (0,2,0) {};
    \node[above right=30pt of unit2, inner sep=0] (unit2label) {$\begin{bmatrix} 0 \\ 0 \\ 1 \\ 0 \end{bmatrix}$};
    \draw[dashed] (unit2) -- (unit2label);

    \node[circle, draw=black, fill=black, inner sep=0.4mm] (unit3) at (0,0,2) {};
    \node[left=10pt of unit3, inner sep=0] (unit3label) {$\begin{bmatrix} 0 \\ 0 \\ 0 \\ 1 \end{bmatrix}$};
    \draw[dashed] (unit3) -- (unit3label);

    \node[circle, draw=black, fill=black, inner sep=0.4mm] (extra2) at (0,1,0) {};
    \node[left=14pt of extra2, inner sep=0] (extra2label) {$\begin{bmatrix} 0 \\ 0 \\ 1/2 \\ 0 \end{bmatrix}$};
    \draw[dashed] (extra2) -- (extra2label);

    \node[circle, draw=black, fill=black, inner sep=0.4mm] (extra1) at (1,0,0) {};
    \node[below=10pt of extra1, inner sep=0] (extra1label) {$\transpose{[0,1/2,0,0]}$};
    \draw[dashed] (extra1) -- (extra1label);

    \end{tikzpicture}
  }%
  \hfill
  \subfloat[%
    The point $(\hat{x},\hat{y}) \in \RLTstrongEF_B(P)$.
    \label{fig_rlt_does_not_dominate_ineq_point}
  ]{%
    $\hat{x} = \begin{bmatrix}
      1/4 \\ 1/4 \\ 1/4 \\ 0
    \end{bmatrix},~ \hat{y} = \begin{bmatrix}
      1/4 & 0 & 0 & 0 \\
      0 & 1/4 & 0 & 0 \\
      0 & 0 & 1/4 & 0 \\
      0 & 0 & 0 & 0 \\
    \end{bmatrix}$ \\[5mm]
  }%

\vspace{2mm}
  \centering%
  \subfloat[%
    RLT inequalities~\eqref{eq_rlt_first} and~\eqref{eq_rlt_second} other than McCormick inequalities.
    \label{fig_rlt_does_not_dominate_ineq_rlt}
  ]{%
    $\begin{aligned}
     &\red{y_{1,j}} &\red{+}\;& \red{y_{2,j}} &\red{+}\;& \red{y_{3,j}} && &\red{\leq}\;& \red{x_j} && \red{\forall j \in N} \\
     &\red{x_1 - y_{1,j}} &\red{+}\;& \red{x_2 - y_{2,j}} &\red{+}\;& \red{x_3 - y_{3,j}} && &\red{\leq}\;& \red{1 - x_j} && \red{\forall j \in N} \\
     &\blue{2y_{1,j}} &\blue{+}\;& \blue{2y_{2,j}} &\blue{+}\;& \blue{2y_{3,j}} &\blue{+}\;& \blue{y_{4,j}} &\blue{\geq}\;& \blue{x_j} && \blue{\forall j \in N} \\
     &\blue{2x_1 - 2y_{1,j}} &\blue{+}\;& \blue{2x_2 - 2y_{2,j}} &\blue{+}\;& \blue{2x_3 - 2y_{3,j}} &\blue{+}\;& \blue{x_4 - y_{4,j}} &\blue{\geq}\;& \blue{1 - x_j} && \blue{\forall j \in N}
    \end{aligned}$
  }%
  \caption{%
    We consider $B = N = \{1,2,3,4\}$ and the polytope \newline
    \begin{minipage}{1.0\textwidth}
      \[
        P \coloneqq \{ x \in [0,1]^N \mid \red{x_1 + x_2 + x_3 + x_4 \leq 1},~ \blue{2x_1 + 2x_2 + 2x_3 + x_4 \geq 1} \}.
      \]
      The section of $P$ at $x_1 = 0$ is depicted in \cref{fig_rlt_does_not_dominate_ineq_picture}.
      The associated disjunction for $D = \{1,2,3\}$ is given by
      $
        x_1 = 1 \lor x_2 = 1 \lor x_3 = 1 \lor (x_1 = x_2 = x_3 = 0),
      $
      that is, $\bar{X}$ consists of the three unit vectors and the origin in $\R^D$.
      It is well-known that $\conv(\bar{X})$ is described by nonnegativity constraints and the inequality $x_1 + x_2 + x_3 \leq 1$, which is valid for $P$.
      Hence, $P \subseteq \conv(\bar{X}) \times [0,1]$.
      The point $\hat{x}$ specified in \cref{fig_rlt_does_not_dominate_ineq_point} lies in $\RLTstrong_B(P)$ since it can be lifted to $(\hat{x},\hat{y}) \in \RLTstrongEF_B(P)$.
      The RLT inequalities~\eqref{eq_rlt_first} and~\eqref{eq_rlt_second} are inequalities~\eqref{eq_mccormick} and those in \cref{fig_rlt_does_not_dominate_ineq_rlt}.
      However, $\hat{x}$ does not lie in the disjunctive hull, which is 
      $\{ x \in [0,1]^4 \mid x_1 + x_2 + x_3 + x_4 = 1 \}$.
    \end{minipage}
    \label{fig_rlt_does_not_dominate_ineq}
  }
\end{figure}

Although the RLT formulation~\eqref{model_qap_aj} in \cref{sec_qap} does not involve linearizations of products of the continuous $w$-variables and the binaries $x$, we were able to establish the dominance in \cref{thmQAPdominanceColsRows}.
We believe that the key property there is that there is a unique optimal $w$-vector for each given $x$-vector, which leaves no room for violating a corresponding convex combination constraint~\eqref{eq_cardinality_equation_convex} for the continuous variables.

We now turn to our characterization, which assumes a complete disjunction on a subset $D$ of the (binary) variables.
We essentially show that the only disjunctions that are dominated for any suitable polytope are those in \cref{thm_cardinality_equation} up to complementing and creating copies of variables.
In particular, when replacing the equation $\sum_{j \in J} x_j = 1$ by the inequality $\sum_{j \in J} x_j \leq 1$, the RLT closure $\RLTstrong_B(\cdot)$ need not dominate the corresponding disjunctive hull.
An example of this is depicted in \cref{fig_rlt_does_not_dominate_ineq}.
The example also shows that the condition $P \subseteq \conv(\bar{X}) \times [0,1]^{N \setminus D}$ is important, i.e., it is not sufficient that the projection $P'$ of $P$ onto $\R^D$ satisfies $P' \cap \Z^D = \bar{X}$.

In order to state our characterization, we introduce the notation
$F^{(j,\beta)} \coloneqq \{ x \in [0,1]^n \mid x_j = \beta \}$ for $j \in N$ and $\beta \in \{0,1\}$ for the faces of the $0/1$-cube.

\begin{theorem}
  \label{thm_rlt_dominates_disjunctions}
  Let $D \subseteq B \subseteq N$ and let $\bar{X} \subseteq \{0,1\}^D$.
  Then
  \begin{equation}
    \RLTstrong_B(P) \subseteq \conv\left( \bigcup_{\bar{x} \in \bar{X}} \{ x \in P \mid x_{D} = \bar{x} \} \right) \label{eq_rltstrong_dominates_disjunctions}
  \end{equation}
  holds for all polyhedra $P \subseteq \conv(\bar{X}) \times [0,1]^{N \setminus D}$ if and only if for every $\bar{x} \in \bar{X}$ there exists a $j(\bar{x}) \in D$ such that $\bar{X} \cap F^{(j(\bar{x}),\beta)}  = \{ \bar{x} \}$ for $\beta = \bar{x}_{j(\bar{x})}$.
\end{theorem}

\begin{proof}
  Assume that for $\bar{X} \subseteq \{0,1\}^D$ the assumption holds for every $\bar{x} \in \bar{X}$.
  First consider the case in which for $\bar{x},x^\star \in \bar{X}$ with $\bar{x} \neq x^\star$ we have some $j(\bar{x}),j(x^\star)\in D$ such that $\bar{X} \cap F^{(j(\bar{x}),\bar{x}_j)} = \{ \bar{x} \}$ and $\bar{X} \cap F^{(j(x^\star),x^\star_j)} = \{ x^\star \}$, with $j(\bar{x})=j(x^\star)$. 
  Since $F^{(j,0)} \cup F^{(j,1)} \supseteq \bar{X}$, this implies $\bar{X} = \{ \bar{x}, x^\star \}$.
  The single-variable disjunction for $x_j$ is dominated by $\mathcal{L}_B(P)$ and from \cref{thm_rlt_lap_dominance} also by $\RLTstrong(P)$ .
  
  Therefore, assume this is not the case.  Then for all $\bar{x},x^\star \in \bar{X}$ we have $j(\bar{x}) \neq j(x^\star)$ if $\bar{x} \neq x^\star$.
  By re-indexing coordinates we can assume that one set of such indices $j$ is $\{1,2,\dotsc,k\}$, where $k = |\bar{X}|$.
  By complementing a subset of the variables $x_1,x_2,\dotsc,x_k$, we can also assume that if there is an $\bar{x} \in \bar{X}$ such that $\bar{X}\cap F^{(j,\beta)} = \{\bar{x}\}$ then $\beta = 1$ for all $j$. 
  By \cref{thm_rlt_elementary}~\ref{thm_rlt_elementary_complement}, this does not change the strength of the RLT. 
  Consequently, for each $j \in \{1,2,\dotsc,k\}$ there is a unique $\bar{x}^{(j)} \in \bar{X}$ with $\bar{x}^{(j)}_j=1$ and $\bar{x}^{(j)}_i = 0$ for all $i \in \{1,2,\dotsc,k\} \setminus \{j\}$.
  Note that all points in $\bar{X}$ are of that form.
  Hence, $\sum_{j=1}^k x_j = 1$ is valid for $\bar{X}$.
  Since $P \subseteq \conv(\bar{X}) \times [0,1]^{N\setminus D}$ holds, this equation is also valid for $P$.
  By \cref{thm_cardinality_equation},
  \begin{equation*}
    \RLTstrong_B(P) \subseteq \conv \left( \bigcup_{j =1}^k \{ x \in P \mid x_j = 1 \} \right) =\conv\left( \bigcup_{\bar{x} \in \bar{X}} \{ x \in P \mid x_{D} = \bar{x} \} \right).  
  \end{equation*}

\medskip

  Conversely, assume there exists some $x^\star \in \bar{X}$ such that for every $i \in D$ there is an $x^{(i)} \in \bar{X} \setminus \{x^\star\}$ for which $x^{(i)}_i = x^\star_i$.
  By complementing we can assume $x^\star = \onevec[D]$ without loss of generality.
  We will construct a polyhedron $P(\alpha) \subseteq \R^N$ for $B \coloneqq N \coloneqq D$ and a point $\hat{x} \in \RLTstrong_B(P(\alpha)) \setminus \conv\left( \bigcup_{\bar{x} \in \bar{X}} \{ x \in P(\alpha) \mid x_{D} = \bar{x} \} \right)$.
  For $\alpha \in [0,1]$, let
  \begin{equation*}
    P(\alpha) \coloneqq \left\{ x \in \conv(\bar{X}) \mid \sum_{i \in D} x_i \leq |D| - \alpha \right\}.
  \end{equation*}
  For each $i \in N$, let $\bar{x}^{(i)} \in F^{(i,1)} \cap \bar{X} \setminus \{x^\star\}$.
  Then, because $\bar{x}^{(i)}$ must differ from $x^\star$ in at least one coordinate, there must exist a $j(i) \in N$ such that $\bar{x}^{(i)}_{j(i)} = 0$.
  Moreover, since $P(0) = \conv(\bar{X})$, we have $x^\star,\bar{x}^{(i)} \in \RLTstrong_B(P(0))$ for all $i \in N$.
  Let $(x^\star,y^\star)$ and $(\bar{x}^{(i)},\bar{y}^{(i)})$ for every $i \in N$ be corresponding solutions in the extended formulation $\RLTstrongEF_B(P(0))$.
  Now consider the convex combination
  \begin{equation*}
    (\hat{x},\hat{y}) \coloneqq (1 - |N| \varepsilon) \cdot \left(x^{\star},y^{\star}\right) + \varepsilon \cdot \sum_{i \in N} \left(\bar{x}^{(i)},\bar{y}^{(i)}\right)
  \end{equation*}
  for $0 \leq \varepsilon \leq \frac{1}{|N|}$.
  By convexity, $(\hat{x},\hat{y}) \in \RLTstrongEF_B(P(0))$ follows.
  By \cref{thm_rlt_characterization} there exist points $\hat{z}^{(i,\beta)} \in \{z\in P(0) \mid z_i = \beta \}$ for all $(i,\beta) \in \Bfrac(\hat{x}) \times \{0,1\}$ such that 
  \begin{align*}
    \hat{z}^{(i,1)}_j \cdot \hat{x}_i - \hat{z}^{(j,1)}_i \cdot \hat{x}_j &= 0,\\
    \hat{z}^{(i,0)}_j \cdot (1-\hat{x}_i) - (1-\hat{z}^{(j,1)}_i) \cdot \hat{x}_j &= 0, 
  \end{align*}
  holds for all $i,j \in \Bfrac(\hat{x})$ and
  \begin{equation*}
    \hat{x}_k = \hat{z}^{(i,\beta)}_k
  \end{equation*}
  holds for all $i \in \Bfrac(\hat{x})$, $\beta=0,1$ and all $k \in N \setminus \Bfrac(\hat{x})$.
  We claim that, for all $(i,\beta) \in \Bfrac(\hat{x}) \times \{0,1\}$ and small enough $\alpha$, also $\hat{z}^{(i,\beta)} \in P(\alpha)$. 
  Because $P(0)$ and $P(\alpha)$ differ in only one constraint, we only need to show that the additional inequality is valid for each $\hat{z}^{(i,\beta)}$.
  First note that from \cref{thm_rlt_elementary}~(\ref{eq_mccormick_ij}) we have $y^*_{i,j} \geq x^\star_i + x_j^\star- 1 = 1$ for all $i,j \in N$, and thus $x^\star_i - y^\star_{j,i} = 0$ for all $i,j \in N$.
  Second, from \cref{thm_rlt_elementary}~(\ref{eq_mccormick_j}) it follows that $\bar{x}_i^{(\ell)} \geq \bar{y}_{j,i}^{(\ell)}$ for all $i,j,\ell \in N$.
  Consequently, for any $i\in \Bfrac(\hat{x})$,
  \begin{multline*}
    \hat{x}_i - \hat{y}_{j(i),i}
    = (1 - |D| \varepsilon) \cdot (x^*_i - y^*_{j(i),i)}) + \varepsilon \cdot \sum_{\ell \in N} \left(\bar{x}_i^{(\ell)} - \bar{y}_{j(i),i}^{(\ell)} \right) \\
    \geq 0 + \varepsilon \cdot \left(\bar{x}_i^{(i)} - \bar{y}_{j(i),i}^{(i)}\right)
    \geq \varepsilon \cdot \left(\bar{x}_i^{(i)} - \bar{x}_{j(i)}^{(i)}\right)
    = \varepsilon.
  \end{multline*}
  Recall that by \cref{rem_rlt_calc_z}, $ \hat{z}^{(i,0)} = \tfrac{\hat{x}-\hat{y}_{\star,i}}{1-\hat{x}_i} $ and $ \hat{z}^{(i,1)} = \tfrac{\hat{y}_{\star,i}}{\hat{x}_i} $ hold.
  Taking coordinate $j(i)$, we get
  \begin{equation*}
    \hat{z}^{(i,1)}_{j(i)} = \frac{\hat{y}_{j(i),i}}{\hat{x}_i} \leq \frac{\hat{y}_{j(i),i}}{\hat{y}_{j(i),i} + \varepsilon} \leq \frac{1}{1+\varepsilon}.
  \end{equation*}
  From $\hat{z}^{(i,\beta)} \in P(0)$ we know that $0 \leq \hat{z}^{(i,\beta)}_j \leq 1$ holds for all $(i,\beta) \in \Bfrac(\hat{x}) \times \{0,1\}$ and all $j \in N$.
  Now let $\alpha \coloneqq \tfrac{\varepsilon}{1+\varepsilon}$.
  We obtain
  \begin{align*}
    \sum_{j \in N} \hat{z}^{(i,1)}_{j}
    &\leq (|N| - 1) + \hat{z}^{(i,1)}_{j(i)}
    \leq |N| - \left(1 - \frac{1}{1 + \varepsilon} \right)
    = |N| - \alpha \text{ and} \\
    \sum_{j \in N} \hat{z}^{(i,0)}_j
    &\leq (|D| - 1) + \hat{z}^{(i,0)}_i
    = |D| - 1
    \leq |D| - \alpha
  \end{align*}
  and thus $\hat{z}^{(i,0)},\hat{z}^{(i,1)} \in P(\alpha)$ for all $i \in \Bfrac(\hat{x})$.
  Now conditions (\ref{eq_rlt_extra1})--(\ref{eq_rlt_fixed}) from \cref{thm_rlt_characterization} are satisfied for the polyhedron $P(\alpha)$ and $\hat{x}$, using $\hat{z}^{(i,\beta)}$ for all $i \in N$ and $\beta = 0,1$. 
  Therefore,
  \begin{equation}
    \hat{x} \in \RLTstrong_B(P(\alpha)). 
    \label{eq_non_dominance_in_rlt}
  \end{equation}
  Now let $\varepsilon < \frac{1}{|N|^2}$.
  Since  $x^\star \notin P(\alpha)$, every point $\bar{x} \in \bar{X} \cap P(\alpha)$ has $\sum_{j \in N} \bar{x}_j \leq |N|-1$.
  Therefore, for every $x' \in \conv\left( \bigcup_{\bar{x} \in \bar{X}} \{ x \in P(\alpha) \mid \sum_{j \in N} x_j = \bar{x} \} \right)$, also $\sum_{j \in N} x'_j \leq |N| - 1$ holds.
  For the point $\hat{x}$, 
  \begin{multline*}
    \sum_{j \in N} \hat{x}_j
    = (1 - |N| \varepsilon) \cdot \sum_{j \in N} x^\star_j + \varepsilon \sum_{j \in N} \sum_{i \in N} x^{(i)}_{j} \\
    \geq (1 - |N|\varepsilon) |N| + 0
    > (1 - \frac{1}{|N|}) |N| = |N| - 1.
  \end{multline*}
  We conclude that
  $
    \hat{x} \notin \conv\left( \bigcup_{\bar{x} \in \bar{X}} \{ x \in P(\alpha) \mid \sum_{j \in N} x_j = \bar{x} \} \right)
    \label{eq_non_dominance_notin_dhull}
  $
  holds, which together with~\eqref{eq_non_dominance_in_rlt} completes the proof.
  \qed
\end{proof}

Since the counter-example in the proof of \cref{thm_rlt_dominates_disjunctions} satisfies $B = N$, we obtain the following necessary condition for dominance of $\RLTweak_B(P)$.

\begin{corollary}
  \label{thm_rlt_dominates_disjunctions_weak}
  Let $D \subseteq B \subseteq N$ and let $\bar{X} \subseteq \{0,1\}^D$.
  If
  \[
    \RLTweak(P) \subseteq \conv\left( \bigcup_{\bar{x} \in \bar{X}} \{ x \in P \mid x_{D} = \bar{x} \} \right)
  \]
  holds for all polyhedra $P \subseteq \conv(\bar{X}) \times [0,1]^{N \setminus D}$ then for every $\bar{x} \in \bar{X}$ there exists a $j(\bar{x}) \in D$ such that $\bar{X} \cap F^{(j(\bar{x}),\beta)}  = \{ \bar{x} \}$ for $\beta = \bar{x}_{j(\bar{x})}$.
\end{corollary}

\section{Conclusions}
\label{sec_conclusions}

In this paper we provided characterization for when a certain point lies in the RLT closure $\RLTweak_B(\cdot)$ in a way that is similar to the geometric interpretation of lift-and-projection together with an extra condition.
While this led to new insights such as \cref{thm_rlt_boundary} and turned out to be useful for other results, we believe that it does not yield complete understanding of the geometry of RLT.
In fact, several questions remain open; in particular it is unclear when multiplication with (complements of) continuous variables adds anything to the strength of the relaxation.

\begin{question}
  For which polyhedra $P \subseteq \R^N$ and $B \subsetneqq N$ is $\RLTstrong_B(P) \subsetneqq \RLTweak_B(P)$?
\end{question}

In \cref{sec_qap} we have shown that the quadratic-assignment-problem formulation~\eqref{model_qap_aj} by Adams and Johnson~\cite{AdamsJ94} dominates the disjunctive hull associated to the cardinality equations of the assignment polytope.
Later in \cref{sec_dominance} we showed a similar result for a general RLT closure.
However, since~\eqref{model_qap_aj} does not arise from an actual formulation of the QAP (but just from the assignment problem's reformulation), \cref{thmQAPdominanceColsRows} is not a consequence of \cref{thm_rlt_dominates_disjunctions}.
Most likely it is possible to generalize our results to also cover such cases, but this may introduce more technicalities, and is thus beyond the scope of this paper.
Despite the dominance of formulations of the form~\eqref{model_qap_cols_rows} by~\eqref{model_qap_aj}, evaluating such formulations in practice remains interesting since decomposition approaches (like Benders' decomposition~\cite{Benders62}) may benefit from the fact that the generation of inequalities~\eqref{eq_qap_x_w_cols_rows} is independent for each $j$ or $i$.

Also our results from \cref{sec_dominance} give rise to new research questions.
The first is on the importance of the requirement $P \subseteq \conv(\bar{X}) \times [0,1]^{N \setminus D}$ for the considered polyhedra $P$, while the second is about sufficiency in \cref{thm_rlt_dominates_disjunctions_weak}.

\begin{question}
  For $\bar{X} = \{ x \in \{0,1\}^D \mid \sum_{j \in D} x_j = 1 \}$, for which polyhedra $P \subseteq \R^N$ whose projection $P' \subseteq \R^D$ on the $x_D$-variables satisfies $P' \cap \Z^D = \bar{X}$ does~\eqref{eq_rltstrong_dominates_disjunctions} hold.
\end{question}

\begin{question}
  Are the conditions in \cref{thm_rlt_dominates_disjunctions_weak} also sufficient for dominance?
\end{question}

The most natural follow-up problem is, in our view, that for a generalization to higher levels of the RLT hierarchy~\cite{SheraliA99}.

\begin{question}
  Which disjunctions are implied by the $k$-th level RLT closure for $k > 1$?
\end{question}

\bibliographystyle{splncs04}
\bibliography{rlt_geometry_disjunctive}

\end{document}